\documentclass[reqno]{amsart}
\usepackage{amsmath}
\usepackage{amsfonts}
\usepackage{amssymb}
\usepackage{amstext}
\usepackage{amsbsy}
\usepackage{amsopn}
\usepackage{amsthm}
\usepackage{amsxtra}
\usepackage{graphicx}
\usepackage{upref}

\newtheorem{theorem}{Theorem}
\newtheorem{lemma}[theorem]{Lemma}
\newtheorem{corollary}[theorem]{Corollary}

\title{A number theoretic question arising in the geometry of plane curves and in billiard dynamics}
\author{Van Cyr}
\address{Department of Mathematics, Northwestern University, Evanston, IL 60208 USA}
\email{cyr@math.northwestern.edu}
\date{\today}
\thanks{}

\begin{document}
\begin{abstract}
We prove that if $\rho\neq1/2$ is a rational number between zero and one, then there is no integer $n>1$ such that
$$
n\tan(\pi\rho)=\tan(n\pi\rho).
$$
This has interpretations both in the theory of bicycle curves and that of mathematical billiards.
\end{abstract}

\maketitle

\section{Introduction}
A closed plane curve $\Gamma:S^1\to\mathbb{R}^2$ of perimeter length $2\pi$ is called a {\em bicycle curve} (of rotation number $\rho$) if $\|\Gamma(t+\rho)-\Gamma(t)\|$ is constant for all $t$ (see the end note in~\cite{T} for a list of papers dealing with bicycle curves).

A theorem of Tabachnikov \cite[Theorem 7]{T} says that the circle admits a non-trivial infinitesimal deformation as a smooth plane bicycle curve of rotation number $\rho$ if and only if $\rho$ is a root of the equation $n\tan(\pi\rho)=\tan(n\pi\rho)$ for some integer $n\geq2$.  Recently, E. Gutkin~\cite{G1} conjectured our Theorem~\ref{mainthm} and showed that it implies that certain billiard maps act like irrational rotations.

We study this equation in the case that $\rho\in\mathbb{Q}$ to determine for which $\rho$ the circle is rigid as a bicycle curve.

A similar trigonometric equation was obtained by Tabachnikov to determine the rigidity of the polynomial analog of a bicycle curve (an ($n,k$)-bicycle polygon; see~\cite{T}).  In~\cite{CC}, R. Connelly and B. Csik\'os studied solutions to the equation and classified the first-order flexible bicycle polygons.

\section{Statement of Results}

Our main result is the following.

\begin{theorem}\label{mainthm}
If $\rho\in(0,1)\cap\mathbb{Q}\setminus\{\frac{1}{2}\}$, then there is no integer $n>1$ such that
$$
n\tan(\pi\rho)=\tan(n\pi\rho).
$$
\end{theorem}

It follows from two lemmas, which we state here.

\begin{lemma}\label{calculation}
Suppose $\rho\in(0,1)\cap\mathbb{Q}\setminus\{\frac{1}{2}\}$, then there exists an integer $n>1$ such that
\begin{equation}\label{identity1}
n\tan(\pi\rho)=\tan(n\pi\rho)
\end{equation}
if and only if
\begin{equation}\label{identity2}
\frac{\sin((n-1)\pi\rho)}{\sin((n+1)\pi\rho)}=\frac{n-1}{n+1}.
\end{equation}
\end{lemma}

\begin{lemma}\label{galois}
If $\rho\in(0,1)\cap\mathbb{Q}\setminus\{\frac{1}{2}\}$ and $k,m\in\mathbb{Z}$ are such that $\sin(m\pi\rho)\neq0$, then
$$
\frac{\sin(k\pi\rho)}{\sin(m\pi\rho)}
$$
is either $-1$, $0$, $1$ or irrational.
\end{lemma}

\begin{proof}[Proof of Theorem~\ref{mainthm}]
By Lemma~\ref{calculation}, any such $n,\rho$ would have to satisfy (\ref{identity2}).  Since $n>1$ we know
$$
\frac{n-1}{n+1}\notin\{-1,0,1\},
$$
so the pair $k:=n-1$, $m:=n+1$ contradicts Lemma~\ref{galois} (that $\sin((n+1)\pi\rho)\neq0$ follows from (\ref{identity2})).
\end{proof}

\section{Proof of Lemma~\ref{calculation}}
For $z\in\mathbb{C}\setminus\big\{\frac{(2k+1)\pi}{2}:k\in\mathbb{Z}\big\}$,
$$
\tan(z)=\frac{i(e^{-iz}-e^{iz})}{e^{-iz}+e^{iz}}.
$$
By assumption $\left|\tan(\pi\rho)\right|<\infty$, so if $n$ satisfies (\ref{identity1}) then $\left|\tan(n\pi\rho)\right|<\infty$.  So our original equation can be rewritten as
$$
n\frac{i(e^{-i\pi\rho}-e^{i\pi\rho})}{e^{-i\pi\rho}+e^{i\pi\rho}}=\frac{i(e^{-in\pi\rho}-e^{in\pi\rho})}{e^{-in\pi\rho}+e^{in\pi\rho}}.
$$
This reduces to
$$
ni(e^{-i\pi\rho}-e^{i\pi\rho})(e^{-in\pi\rho}+e^{in\pi\rho})=i(e^{-in\pi\rho}-e^{in\pi\rho})(e^{-i\pi\rho}+e^{i\pi\rho}),
$$
which further simplifies to
\begin{equation}\label{identity3}
(n-1)\big(e^{-i(n+1)\pi\rho}-e^{i(n+1)\pi\rho}\big)=(n+1)\big(e^{-i(n-1)\pi\rho}-e^{i(n-1)\pi\rho}\big).
\end{equation}
Since $n>1$ we know $(n-1)(n+1)\neq0$, so if $e^{-i(n+1)\pi\rho}-e^{i(n+1)\pi\rho}=0$ then $e^{-i(n-1)\pi\rho}-e^{i(n-1)\pi\rho}=0$.  But if $e^{-i(n-1)\pi\rho}=e^{i(n-1)\pi\rho}$, then
$$
e^{-i(n+1)\pi\rho}=e^{-i(n-1)\pi\rho}e^{-2\pi i\rho}=e^{i(n-1)\pi\rho}e^{-2\pi i\rho}\neq e^{i(n-1)\pi\rho}e^{2\pi i\rho}=e^{i(n+1)\pi\rho}
$$
which is a contradiction (here we used $\rho\neq\frac{1}{2}$).  So we can divide in (\ref{identity3}) to get
$$
\frac{n-1}{n+1}=\frac{e^{-i(n-1)\pi\rho}-e^{i(n-1)\pi\rho}}{e^{-i(n+1)\pi\rho}-e^{i(n+1)\pi\rho}}=\frac{\big(\frac{e^{i(n-1)\pi\rho}-e^{-i(n-1)\pi\rho}}{2i}\big)}{\big(\frac{e^{i(n+1)\pi\rho}-e^{-i(n+1)\pi\rho}}{2i}\big)}=\frac{\sin((n-1)\pi\rho)}{\sin((n+1)\pi\rho)}.
$$
$\hfill\square$

\section{Proof of Lemma~\ref{galois}}

In this section we set $\omega_n:=e^{2\pi i/n}$ to be a primitive $n^{th}$ root of unity and $\mathbb{Q}(\omega_n)$ the $n^{th}$ cyclotomic field.  We need the following two well-known facts (see~\cite{W}, for example):
\begin{enumerate}
\item $\left|\mathbb{Q}(\omega_n):\mathbb{Q}\right|=\phi(n)$ (Euler's $\phi$-function);
\item $\mathbb{Q}(\omega_n)=Span\{1,\omega_n,\omega_n^2,\dots,\omega_n^{n-1}\}$.
\end{enumerate}

\subsection{Basic Strategy of Proof}

We want to show that if $\rho=\frac{p}{q}$ and $\sin(k\pi\rho)=\lambda\sin(m\pi\rho)$ then $\lambda\in\{-1,0,1\}$.  Since $\sin(k\pi\rho), \sin(m\pi\rho)\in\mathbb{Q}(\omega_{2q})$, our main tool is the following simple (but useful) observation.

\begin{lemma}\label{linind}
Let $\mathcal{B}$ be a basis for an $n$-dimensional $\mathbb{Q}$-vector space $V$ and suppose $u, v\in V$ are vectors whose coordinates (relative to $\mathcal{B}$) all come from the set $\{-1,0,1\}$.  If $u=\lambda v$ for some $\lambda\in\mathbb{Q}$, then $\lambda\in\{-1,0,1\}$.
\end{lemma}

We will prove Lemma~\ref{galois} by (explicitly) constructing a basis $\mathcal{B}$ for $\mathbb{Q}(\omega_{2q})$ in which, for every integer $\ell$, $i\Im(\omega_{2q}^{\ell})$ is of the type described in Lemma~\ref{linind}.

\subsection{Motivating Case}

This subsection is not necessary to prove Lemma~\ref{galois}, but is included to demonstrate the main idea of the proof.  Here we will prove Lemma~\ref{galois} in the special case that the denominator of $\rho$ is an odd prime (see Corollary~\ref{cor1}).

\begin{lemma}\label{prime}
If $n$ is an odd prime and we define
\begin{eqnarray*}
A_n&:=&\{\phantom{i}\Re(\omega_n), \phantom{i}\Re(\omega_n^2), \dots, \phantom{i}\Re(\omega_n^{(n-1)/2})\} \\
B_n&:=&\{i\Im(\omega_n), i\Im(\omega_n^2), \dots, i\Im(\omega_n^{(n-1)/2})\},
\end{eqnarray*}
then $A_n\cup B_n$ is a basis for $\mathbb{Q}(\omega_n)$ over $\mathbb{Q}$.
\end{lemma}
\begin{proof}
First note that $\omega_n^k\in Span(A_n)$ for every $k=1,2,\dots,n$ (sum the geometric series $\omega_n+\omega_n^2+\cdots+\omega_n^{n-1}=-1$ to see the $k=n$ case).  Next $\left|A_n\cup B_n\right|=n-1=\phi(n)$ (since $n$ is prime).  Any spanning set with $\phi(n)$ elements is a basis.
\end{proof}

Since $\left|\mathbb{Q}(\omega_{2n}):\mathbb{Q}\right|=\phi(2n)=\phi(n)=\left|\mathbb{Q}(\omega_n):\mathbb{Q}\right|$, Lemma~\ref{prime} also produces a basis for $\mathbb{Q}(\omega_{2n})$.

\begin{corollary}[Lemma~\ref{galois} for an odd prime]\label{cor1}
If $n$ is an odd prime and $k_1, k_2\in\mathbb{Z}$ are such that $\sin\big(\frac{k_2\pi}{n}\big)\neq0$, then
\begin{equation}\label{calc}
\frac{\sin\big(\frac{k_1\pi}{n}\big)}{\sin\big(\frac{k_2\pi}{n}\big)}
\end{equation}
is either $-1, 0, 1$ or irrational.
\end{corollary}
\begin{proof}
There exists $0<\tilde{k}_1\leq\tilde{k}_2\leq\frac{n-1}{2}$ such that
\begin{eqnarray*}
\sin\Big(\frac{k_1\pi}{n}\Big)&=&\pm\Im(\omega_{2n}^{\tilde{k}_1}) \\
\sin\Big(\frac{k_2\pi}{n}\Big)&=&\pm\Im(\omega_{2n}^{\tilde{k}_2}).
\end{eqnarray*}
If~\ref{calc} is a rational number, say $i\Im(\omega_{2n}^{\tilde{k}_1})=i\lambda\Im(\omega_{2n}^{\tilde{k}_2})$ for some $\lambda\in\mathbb{Q}$, then $\lambda\in\{-1,0,1\}$ (by Lemma~\ref{linind}).
\end{proof}

\subsection{Prime Powers}

Here we generalize Lemma~\ref{prime} to the case where the denominator of $\rho$ is a prime power.

\begin{lemma}\label{primepower}
If $n=p^k$ is an odd prime power and we define
\begin{eqnarray*}
A_{p^k}&:=&\{\phantom{i}\Re(\omega_n^t):1\leq t\leq\phi(n)/2\} \\
B_{p^k}&:=&\{i\Im(\omega_n^t):1\leq t\leq\phi(n)/2\},
\end{eqnarray*}
then $A_n\cup B_n$ is a basis for $\mathbb{Q}(\omega_n)$ over $\mathbb{Q}$.  Moreover for any integer $t$,  all coefficients of the vectors $\Re(\omega_n^t)$ and $i\Im(\omega_n^t)$ (with respect to this basis) are contained in the set $\{-1,0,1\}$.
\end{lemma}
\begin{proof}
The set $A_n\cup B_n$ contains exactly $\phi(n)=p^{k-1}(p-1)$ many elements so to prove the first claim it suffices to show that $Span\{A_n\cup B_n\}=\mathbb{Q}(\omega_n)$, or simply that $\omega_{p^k}^t\in Span\{A_n\cup B_n\}$ for every integer $t\in[0,p^k-1]$.

It is immediate that if $1\leq t\leq\frac{p^{k-1}(p-1)}{2}$, then $\omega_{p^k}^{\pm t}\in Span\{A_n\cup B_n\}$ (recall that $\Re(\omega_n^{-t})=\Re(\omega_n^t)$ and $\Im(\omega_n^{-t})=-\Im(\omega_n^t)$).  That is, the only integers $t\in[1,p^k-1]$ for which we have not yet verified that $\omega_{p^k}^t\in Span\{A_n\cup B_n\}$ are those satisfying 
$$
\frac{p^{k-1}(p-1)}{2}<t<\frac{p^{k-1}(p-1)}{2}+p^{k-1}.
$$
If $0<s<p^{k-1}$, then there is precisely one $r\in[0,(p-1)/2)$ so that 
$$
t_s:=rp^{k-1}+s+1\in\bigg(\frac{p^{k-1}(p-1)}{2},\frac{p^{k-1}(p-1)}{2}+p^{k-1}\bigg).
$$
On the other hand, for any $s$,
\begin{equation}\label{geometricseries}
\omega_{p^{k+1}}^s+\omega_{p^{k+1}}^{p^{k-1}+s}+\omega_{p^{k+1}}^{2p^{k-1}+s}+\cdots+\omega_{p^{k+1}}^{(p-1)p^{k-1}+s}=0
\end{equation}
(divide both sides by $\omega_{p^{k+1}}^s$ and sum the geometric series).  
In light of (\ref{geometricseries}) and our previous verification that $\omega_n^t\in Span\{A_n\cup B_n\}$ for every $t\equiv s$ (mod $p^{k-1}$) (except $t_s$ itself), we see $\omega_n^{t_s}\in Span\{A_n\cup B_n\}$.  It remains only to see that $1=\omega_n^0\in Span\{A_n\cup B_n\}$.  This follows because from (\ref{geometricseries}) with $s=0$ and the observation that $rp^{k-1}\notin\big(\frac{p^{k-1}(p-1)}{2},\frac{p^{k-1}(p-1)}{2}+p^{k-1}\big)$ for any integer $r$.

The second claim follows from the previous two paragraphs.
\end{proof}

\begin{lemma}\label{evenpower}
If $n=2^k$ and we define (for integers $t$)
\begin{eqnarray*}
A_{2^k}&:=&\{\phantom{i}\Re(\omega_n^t):0\leq t<2^{k-2}\} \\
B_{2^k}&:=&\{i\Im(\omega_n^t):0<t\leq2^{k-2}\},
\end{eqnarray*}
then $A_n\cup B_n$ is a basis for $\mathbb{Q}(\omega_n)$ over $\mathbb{Q}$.  Moreover for any integer $t$,  all coefficients of the vectors $\Re(\omega_n^t)$ and $i\Im(\omega_n^t)$ (with respect to this basis) are contained in the set $\{-1,0,1\}$.
\end{lemma}
\begin{proof}
As before, the number of elements in $A_{2^k}\cup B_{2^k}$ is $2^{k-1}=\phi(2^k)$ so it suffices to see that $\omega_n^t\in Span(A_{2^k}\cup B_{2^k})$ for every $0\leq t<2^k$.

It is immediate that $\omega_n^t\in Span(A_{2^k}\cup B_{2^k})$ for $0\leq t\leq2^{k-2}$ (note that $\Re(\omega_n^{2^{k-2}})=\Im(\omega_n^0)=0$).  This describes the set of all $(2^k)$-th roots of unity in the first quadrant in $\mathbb{C}$.  The set of all roots is symmetric about the real and imaginary axes, so $\omega_n^t\in Span(A_{2^k}\cup B_{2^k})$ for every $t$.

Again, the second claim is by construction.
\end{proof}

%\begin{corollary}\label{cor2}
%If $n=p^k$ is a prime power and $k_1, k_2\in\mathbb{Z}$ are such that $\sin\big(\frac{k_2\pi}{n}\big)\neq0$, then
%\begin{equation}\label{calc2}
%\frac{\sin\big(\frac{k_1\pi}{n}\big)}{\sin\big(\frac{k_2\pi}{n}\big)}
%\end{equation}
%is either $-1, 0, 1$ or irrational.
%\end{corollary}
%\begin{proof}
%As in Corollary~\ref{cor1}, if (\ref{calc2}) is not equal to zero, then by Lemmas~\ref{primepower} and~\ref{evenpower} all of the coefficients of the vectors $\sin\big(\frac{k_1\pi}{n}\big)=\Im(\omega_n^{k_1})$ and $\sin\big(\frac{k_2\pi}{n}\big)=\Im(\omega_n^{k_2})$ (with respect to the $(A_{p^k}\cup B_{p^k})$-basis) are $-1, 0$ or $1$.  If these vectors are not ($\mathbb{Q}$-)linearly independent, then they differ by a factor of $\pm1$.
%\end{proof}

\subsection{General Result}

Finally we are ready to prove Lemma~\ref{galois} in full generality.  We begin with some notation: if $S_1, S_2, \dots, S_n$ are (nonempty) subsets of $\mathbb{C}$, define
$$
S_1S_2\cdots S_n:=\{\alpha_1\alpha_2\cdots\alpha_n\in\mathbb{C}:\alpha_i\in S_i\}.
$$
\begin{lemma}\label{general}
If $p_1^{e_1}p_2^{e_2}\cdots p_k^{e_k}$ is the prime factorization of an integer $n$ and we define, for $1\leq i\leq k$ and $0\leq j\leq 1$,
$$
D_i^j=\left\{\hspace{-0.05 in}\begin{tabular}{ll}$A_{q_i}$ & if $j=0$ \\ $B_{q_i}$ & if $j=1$\end{tabular}\right.
$$
and set
$$
D_n:=\bigcup_{(j_1,\dots,j_k)\in\{0,1\}^k}D_1^{j_1}\cdots D_k^{j_k},
$$
then $D_n$ is a basis for $\mathbb{Q}(\omega_n)$ over $\mathbb{Q}$.  Moreover for any integer $t$,  all coefficients of the vectors $\Re(\omega_n^t)$ and $i\Im(\omega_n^t)$ (with respect to this basis) are contained in the set $\{-1,0,1\}$.
\end{lemma}
\begin{proof}
It is immediate that $D_n$ is a basis for $\mathbb{Q}(\omega_n)$, but it remains to see that all coefficients (with respect to the basis $D_n$) of the vectors $\Re(\omega_n^t)$ and $i\Im(\omega_n^t)$ are in the set $\{-1,0,1\}$.

Suppose $0\leq t<n$ and write
$$
t=s_1(n/q_1)+s_2(n/q_2)+\cdots+s_k(n/q_k),
$$
where $0\leq s_i<q_i$ for $i=1,2,\dots,k$.  Then $\omega_n^t=\omega_{q_1}^{s_1}\omega_{q_2}^{s_2}\cdots\omega_{q_k}^{s_k}$.  Write
\begin{eqnarray*}
\Re(\omega_{q_i}^{s_i})&=&\lambda_{i,1}\Re(\omega_{q_i})+\lambda_{i,2}\Re(\omega_{q_i}^2)+\cdots+\lambda_{i,\phi(q_i)/2}\Re(\omega_{q_i}^{\phi(q_i)/2}) \\
\Im(\omega_{q_i}^{s_i})&=&\mu_{i,1}\Im(\omega_{q_i})+\mu_{i,2}\Im(\omega_{q_i}^2)+\cdots+\mu_{i,\phi(q_i)/2}\Im(\omega_{q_i}^{\phi(q_i)/2}),
\end{eqnarray*}
where $\lambda_{i,j}, \mu_{i,j}\in\{-1,0,1\}$ (possible by Lemmas~\ref{primepower} and~\ref{evenpower}).  Then $\Re(\omega_n^t)$ and $i\Im(\omega_n^t)$ are both sums of expressions of the form
$$
\pm\delta_1\delta_2\cdots\delta_k\mathcal{G}_1\mathcal{G}_2\cdots\mathcal{G}_k,
$$
where $\mathcal{G}_i\in A_{q_i}\cup B_{q_i}$ and $\delta_i\in\{\lambda_{i,1},\dots,\lambda_{i,\phi(q_i)/2},\mu_{i,1},\dots,\mu_{i,\phi(q_i)/2}\}$.  The result follows.
%\begin{eqnarray*}
%\delta_i&\in&\{\lambda_{i,1},\dots,\lambda_{i,\phi(q_i)/2},\mu_{i,1},\dots,\mu_{i,\phi(q_i)/2}\}, \\
%\mathcal{G}_i&\in&\{\Re(\omega_{q_i}),\dots,\Re(\omega_{q_i}^{\phi(q_i)/2}),\Im(\omega_{q_i}),\dots,\Im(\omega_{q_i}^{\phi(q_i)/2})\}.
%\end{eqnarray*}
%Since $\mathcal{G}_1\mathcal{G}_2\cdots\mathcal{G}_k\in D_n$ and $\delta_1\delta_2\cdots\delta_k\in\{-1,0,1\}$, the result follows.
\end{proof}
\begin{proof}[Proof of Lemma~\ref{galois}]
Let $\rho\in(0,1)\cap\mathbb{Q}\setminus\{\frac{1}{2}\}$ and suppose $k,m\in\mathbb{Z}$ are such that $\sin(m\pi\rho)\neq0$.  If $\sin(k\pi\rho)/\sin(m\pi\rho)=\lambda\in\mathbb{Q}$ then, by Lemmas~\ref{linind} and~\ref{general}, $\lambda\in\{-1,0,1\}$.
\end{proof}

\end{document}